\documentclass[11pt,twoside,reqno]{amsart}
\usepackage[utf8]{inputenc}
\usepackage[T1]{fontenc}
\usepackage{import}
\usepackage{newclude}
\usepackage{amsmath,amssymb,amsthm,mathtools}
\usepackage[toc,page]{appendix}
\usepackage{longtable,booktabs,array}
\usepackage{anysize}
\usepackage{amscd}
\usepackage{stmaryrd}
\usepackage{wasysym}
\usepackage{mathrsfs}
\usepackage[scr=boondox]{mathalfa}
\usepackage{enumitem}
\usepackage{accents}
\usepackage{dsfont}
\usepackage{soul}
\usepackage{color}
\usepackage[all]{xy}
\usepackage{float}
\usepackage{bm}
\usepackage{makecell}
\usepackage{thmtools}      
\usepackage{setspace}
\usepackage{comment}
\usepackage{tikz}
\usepackage{tikz-cd}
\usepackage{mathdots}
\usepackage{mathtools}
\usepackage{graphicx}
\usepackage[myheadings]{fullpage}
\usepackage{bbm}

\usepackage[hypertexnames=false,colorlinks,linkcolor=black,citecolor=black]{hyperref}
\usepackage[capitalize,nameinlink]{cleveref}

\crefname{appendix}{Appendix}{Appendices}
\Crefname{appendix}{Appendix}{Appendices}

\numberwithin{equation}{section}
\newcommand\mtop{.95in}
\newcommand\mbottom{.95in}
\newcommand\mleft{1in}
\newcommand\mright{1in}
\usepackage[top = \mtop, bottom = \mbottom, left = \mleft, right=\mright]{geometry}

\DeclareMathOperator{\Mat}{Mat}

\newtheorem{thm}{Theorem}[section]

\newtheorem{prop}[thm]{Proposition}

\newtheorem{lemma}[thm]{Lemma}
\newtheorem{conj}[thm]{Conjecture}

\theoremstyle{definition}
\newtheorem{defi}[thm]{Definition}

\newtheorem{rmk}[thm]{Remark}

\newcommand\reallywidehat[1]{%
\savestack{\tmpbox}{\stretchto{%
  \scaleto{%
    \scalerel*[\widthof{\ensuremath{#1}}]{\kern-.6pt\bigwedge\kern-.6pt}%
    {\rule[-\textheight/2]{1ex}{\textheight}}
  }{\textheight}%
}{0.5ex}}%
\stackon[1pt]{#1}{\tmpbox}%
}
\DeclareSymbolFont{bbold}{U}{bbold}{m}{n}
\DeclareSymbolFontAlphabet{\mathbbold}{bbold}

\makeatletter
\def\@tocline#1#2#3#4#5#6#7{\relax
  \ifnum #1>\c@tocdepth 
  \else
    \par \addpenalty\@secpenalty\addvspace{#2}%
    \begingroup \hyphenpenalty\@M
    \@ifempty{#4}{%
      \@tempdima\csname r@tocindent\number#1\endcsname\relax
    }{%
      \@tempdima#4\relax
    }%
    \parindent\z@ \leftskip#3\relax \advance\leftskip\@tempdima\relax
    \rightskip\@pnumwidth plus4em \parfillskip-\@pnumwidth
    #5\leavevmode\hskip-\@tempdima
      \ifcase #1
       \or\or \hskip 1em \or \hskip 2em \else \hskip 3em \fi%
      #6\nobreak\relax
    \hfill\hbox to\@pnumwidth{\@tocpagenum{#7}}\par
    \nobreak
    \endgroup
  \fi}
\makeatother

\newcommand{\R}{\mathbb{R}}

\newcommand{\Z}{\mathbb{Z}}

\newcommand{\A}{\mathfrak{A}}
\renewcommand{\S}{\mathfrak{S}}

\newcommand{\C}{\mathbb{C}}

\newcommand{\E}{\mathbb{E}}

\newcommand{\mc}{\mathcal}

\newcommand{\bbone}{\mathbbold{1}}
\renewcommand{\l}{\lambda}

\renewcommand{\L}{\mathcal{L}}

\DeclareMathOperator{\Tr}{Tr}

\DeclareMathOperator{\Id}{Id}

\DeclareMathOperator{\Ind}{Ind}
\DeclareMathOperator{\Res}{Res}

\DeclareMathOperator{\sgn}{sgn}

\DeclareMathOperator{\sym}{sym}

\DeclareMathOperator{\Pois}{Pois}

\title{The $k$-cycle shuffling with repeated cards}
\author{Jiahe Shen}

\date{\today}

\begin{document}

\thanks{The author thanks Evita Nestoridi for helpful discussions, and acknowledges the Simons Center for Geometry and Physics for funding the Columbia/Stony Brook Probability Day where these discussions occurred. The author also thanks Ivan Corwin and Roger Van Peski for reading the draft and providing comments. The author acknowledges support from Ivan Corwin's NSF grant DMS-2246576 and Simons Investigator grant 929852.}

\maketitle

\begin{abstract}
We investigate the $k$-cycle shuffle on repeated cards, namely on a deck consisting of $l$ identical copies of each of $m$ card types, with total size $n=ml$. We establish asymptotic results for the total variation mixing of this shuffle, including cutoff and explicit limiting profiles. For fixed $l$, we show that the walk exhibits cutoff at time $\frac{n}{k}\log n$ with window of order $\frac{n}{k}$, and we identify the limiting profile in terms of the total variation distance between Poisson distributions arising from quotient fixed-point statistics. When $l\to\infty$ with sufficiently slow growth, more precisely when $l=o(\log n)$, we prove that the cutoff location shifts to $\frac{n}{k}\left(\log n-\frac 12\log l\right)$, again with window of order $\frac{n}{k}$, and that the limiting profile is asymptotically Gaussian, arising from a Poisson comparison after normal approximation.

The proof is based on an approximation of the shuffling measure by an explicitly tractable auxiliary measure, generalizing the $k=2$ case from Jain and Sawhney \cite{jain2024hitting}. The representation-theoretic framework underlying the analysis of this auxiliary measure follows from the work of Hough \cite{hough2016random} and Nestoridi and Olesker-Taylor \cite{nestoridi2022limit}.
\end{abstract}



\tableofcontents

\section{Introduction}\label{sec: Intro}

\subsection{Main results}

Questions about how repeated shuffling randomizes a deck of cards naturally lead to the study of random walks on the symmetric group. Roughly speaking, the cutoff phenomenon means that convergence to randomness happens abruptly: the system remains far from equilibrium for a long time, and then becomes well mixed within a much shorter time window. A landmark result of Diaconis--Shahshahani \cite{diaconis1981generating} showed that the random transposition shuffle exhibits cutoff at time $\frac12 n\log n$, thereby introducing representation-theoretic techniques that have since become central in the subject. Much later, Teyssier \cite{teyssier2020limit} determined the corresponding limiting profile, proving that at time $\frac12 n(\log n+c)$ the total variation distance converges to
\[
d_{TV}\!\left(\Pois(1+\exp(-c)),\Pois(1)\right).
\]
For the random \(k\)-cycle shuffle, Berestycki, Schramm, and Zeitouni \cite{MR2884874} established cutoff at time \(\frac{n}{k}\log n\) when \(k\) is fixed. Hough \cite{hough2016random} later extended the cutoff result to the regime \(2\le k\le o(n/\log n)\) by developing refined character estimates for the symmetric group. Building on the same character-estimate framework, Nestoridi and Olesker-Taylor \cite{nestoridi2022limit}, together with Teyssier's approximation method, determined the limiting profile in the same regime: at time \(\frac{n}{k}(\log n+c)\), the total variation distance converges to
\[
d_{TV}\!\left(\Pois(1+\exp(-c)),\Pois(1)\right).
\]
More recently, Olesker-Taylor, Teyssier, and Thévenin \cite{olesker2025sharp} developed sharp character bounds in a substantially more general setting, strong enough to recover cutoff for the full range of \(k\) and to suggest corresponding profile results as well. 

Based on the above results, it is natural to consider the analogous shuffling problem when the deck contains repeated cards, namely when there are $l$ indistinguishable copies of each of $m$ card types. This leads to a Markov chain on a quotient of the symmetric group, a viewpoint emphasized for instance in the book of Diaconis-Fulman \cite[Chapter 4]{diaconis2023mathematics}. The goal of the present paper is to determine the cutoff time and limiting profile for the random $k$-cycle shuffle in this repeated-card setting, in a form that recovers the classical single-deck results when $l=1$. Before stating our main results, we first introduce the basic definitions and notation.
For all $n\ge 1$, let $\S_n$ be the symmetric group over $n$ symbols, $\A_n\subset\S_n$ be the alternating group, and $\A_n^c\subset\S_n$ be the complement of the alternating group. Denote by $U_{\S_n},U_{\A_n},U_{\A_n^c}$ the uniform distribution over $\S_n,\A_n,\A_n^c$, respectively. 

\begin{defi}\label{defi: quotient of space}
Suppose $n=ml$ where $m,l\ge 1$. Then we abbreviate $\sim_l\backslash\S_n$ for the set of left cosets
$$(\underbrace{\S_l\times\cdots\times \S_l}_{m\text{ times}})\backslash\S_n:=\{(\underbrace{\S_m\times\cdots\times \S_m}_{l\text{ times}})\circ\sigma:\sigma\in\S_n\},$$
where 
$$\underbrace{\S_l\times\cdots\times \S_l}_{m\text{ times}}:=\{\sigma\in\S_n:\sigma(i)\equiv i\pmod m\quad\forall1\le i\le n\}$$ 
is a subgroup of $\S_n$. 
Denote by $\sim_l\backslash U_{\S_n}$ the uniform distribution over $\sim_l\backslash\S_n$.
\end{defi}

Now we give the definition of a random $k$-cycle. Our setup involves parity constraints which follow from Nestoridi and Olesker-Taylor \cite{nestoridi2022limit} and Hough \cite{hough2016random}:

\begin{defi}\label{defi: Pnk and its convolution}
Let $P_{n,k}$ be the random permutation in $\S_n$ given by
\begin{equation}\label{eq: measure of k-cycle}
P_{n,k}:=\begin{cases}
(i_1\ldots i_k) & \text{with probability }\frac{k}{n(n-1)\cdots(n-k+1)}\\
\text{otherwise} & \text{with probability zero}\end{cases}.
\end{equation}
Moreover, for any $t\ge 1$, let $P_{n,k}^{*t}$ be the $t$-fold convolution of $P_{n,k}$. If $k$ is even and $t$ is odd, we view $P_{n,k}^{*t}$ as a random variable in $\A_n^c$; otherwise, we view $P_{n,k}^{*t}$ as a random variable in $\A_n$. Given $n=ml$ with $l\ge 2$, the random variable $\sim_l\backslash P_{n,k}^{*t}\in\sim_l\backslash\S_n$ is induced from \Cref{defi: quotient of space} in the obvious way.
\end{defi}

It is worth highlighting that in the above definition, we regard $P_{n,k}^{*t}$ as a random variable in $\A_n$ or its complement, but we regard $\sim_l\backslash P_{n,k}^{*t}\in\sim_l\backslash\S_n$ as a random variable in the quotient space of $\S_n$. 

We begin with the regime in which the multiplicity parameter \(l\) grows with \(n\). In this case, the repeated-card structure already changes the location of cutoff and leads to a different limiting profile from the classical one.

\begin{thm}[\Cref{thm: restate of window growing m} in the text]\label{thm: window growing m}
Let \(n=ml\), such that \(l=\omega(1)\) but \(l=o(\log n)\). Suppose \(2\le k\le n^{c_0+o(1)}\), where \(c_0\in[0,1)\) is an absolute constant, and
$$
t=\frac nk\left(\log n-\frac 12\log l+c\right),
$$
where \(c\in \R\) is an absolute constant. Then, we have
$$
d_{TV}(\sim_l\backslash P_{n,k}^{*t},\sim_l\backslash U_{\S_n})
=
2\Phi(\exp(-c)/2)-1+o(1).
$$
Here, \(\Phi(x):=\frac{1}{\sqrt{2\pi}}\int_{-\infty}^x\exp(-y^2/2)\,dy\) denotes the cumulative distribution function of the normal distribution \(\mathcal{N}(0,1)\).
\end{thm}

\Cref{thm: window growing m} shows that once the number of repeated copies tends to infinity, the chain mixes earlier than in the classical \(k\)-cycle shuffle: compared with the usual cutoff time \(\frac{n}{k}\log n\), the cutoff is advanced by \(\frac{n}{2k}\log l\). The window remains of order \(\frac{n}{k}\), but the limiting profile is now Gaussian rather than Poissonian. As far as we know, this is a new phenomenon in the study of random \(k\)-cycle shuffles.

We next turn to the fixed-multiplicity regime. In that case, the cutoff location is the same as in the classical setting, but the limiting profile is shifted by the repeated-card structure.

\begin{thm}[\Cref{thm: restate of window fixed m} in the text]\label{thm: window fixed m}
Let \(n=ml\), such that \(l\ge 2\) is fixed, and \(m=\omega(1)\). Suppose \(2\le k\le o(n/\log n)\), and
$$
t=\frac nk(\log n+c),
$$
where \(c\in \R\) is an absolute constant. Then, we have
$$
d_{TV}(\sim_l\backslash P_{n,k}^{*t},\sim_l\backslash U_{\S_n})
=
d_{TV}(\Pois(l+\exp(-c)),\Pois(l))+o(1).
$$
Here, \(\Pois(\cdot)\) denotes the Poisson distribution.
\end{thm}

We emphasize that \Cref{thm: window growing m} and \Cref{thm: window fixed m} concern convergence to the uniform distribution on the quotient space \(\sim_l\backslash \S_n\), rather than to the uniform distribution on \(\A_n\) or on \(\A_n^c\). This is a feature specific to the repeated-card setting. Nevertheless, when \(l=1\), the conclusion of \Cref{thm: window fixed m} parallels the classical \(k\)-cycle result: the cutoff occurs at time \(\frac{n}{k}\log n\) with a window of order \(\frac{n}{k}\), and the limiting profile is exactly
$$
d_{TV}\!\left(\Pois(1+\exp(-c)),\Pois(1)\right).
$$
To the best of our knowledge, \Cref{thm: window growing m} and \Cref{thm: window fixed m} provide the first explicit determination of both the cutoff location and the limiting profile for the random \(k\)-cycle shuffle on repeated cards. In particular, they exhibit a transition from a Poissonian profile to a Gaussian one as the multiplicity parameter \(l\) grows.

We now briefly outline the proof of our main results. A key step is inspired by the work of Jain and Sawhney \cite[Theorem 1.3]{jain2024hitting}, which treats the case $k=2$. More precisely, we construct another probability measure on the shuffling group that follows a completely different sampling procedure from the original random $k$-cycle walk, but is considerably more amenable to analysis. The advantage of this auxiliary measure is that its quotient fixed-point statistics can be described explicitly, which ultimately allows us to extract the cutoff location and limiting profile.

\begin{defi}\label{defi: construction of nu}
Suppose $t\in\Z$. Set $t':=t-\lfloor\frac{n\log n}{k}\rfloor$, and $\gamma_t:=\exp(-kt'/n)$. Let $\nu_{n,k}^t$ denote the measure on $\S_n$ defined by the following sampling process: first, sample $M_t\in\{0,1,\ldots,n\}$ according to the distribution $\mathbf{P}[M_t=x]=\mathbf{P}(\Pois(\gamma_t)=x)/\mathbf{P}(\Pois(\gamma_t)\le n)$, then sample a uniformly random subset $S_t$ of size $M_t$ in $[n]$. Finally,
\begin{enumerate}
\item When $k$ is odd, sample a uniformly random element of $\A_{[n]\backslash S_t}$ and view it as an element of $\A_n$ by fixing all of the elements in $S_t$.
\item When $k$ is even, and $t$ is odd, sample a uniformly random element of $\A_{[n]\backslash S_t}^c$ and view it as an element of $\A_n^c$ by fixing all of the elements in $S_t$.
\item When $k$ is even, and $t$ is even, sample a uniformly random element of $\A_{[n]\backslash S_t}$ and view it as an element of $\A_n$ by fixing all of the elements in $S_t$.
\end{enumerate}
\end{defi}

The next theorem shows that the original shuffling measure can be well approximated by the auxiliary measure $\nu_{n,k}^t$, in the sense that their total variation distance is small in the regimes relevant to our main results.

\begin{thm}[\Cref{thm: restate of pnkt and nu} in the text]\label{thm: pnkt and nu}
For $n=\omega(1)$ and $k\ge 2$, the following holds. Let $t':=t-\frac{n\log n}{k}$. Let $P_{n,k}^{*t}$ be defined as in \Cref{defi: Pnk and its convolution}, and $\nu_{n,k}^t$ be defined as in \Cref{defi: construction of nu}. 
\begin{enumerate}
\item Suppose that $2\le k\le o(n/\log n)$, and $t'=cn/k$, where $c\in\R$ is an absolute constant. Then, we have $$d_{TV}(P_{n,k}^{*t},\nu_{n,k}^t)=o(1).$$ 
\item Suppose that $2\le k\le n^{c_0+o(1)}$, where $c_0\in[0,1)$ is an absolute constant, and $|t'|\le \frac{n(\log\log n-\omega(1))}{2k}$.
Then, we have $$d_{TV}(P_{n,k}^{*t},\nu_{n,k}^t)\le n^{c_0-1+o(1)}.$$ 
\end{enumerate}
\end{thm}

With \Cref{thm: pnkt and nu} in hand, it remains to compare the auxiliary measure $\nu_{n,k}^t$ with the uniform probability measure on the quotient space. We first reduce this comparison to the law of the set of quotient fixed points. By symmetry, conditional on the number of quotient fixed points, both measures are uniform over all configurations with that cardinality, so the problem further reduces to comparing the distribution of the number of quotient fixed points. Since it is well known that, under the uniform measure, the number of quotient fixed points converges to a Poisson distribution, our task is thereby reduced to comparing the total variation distance between two Poisson laws with different parameters, which leads to the desired limiting profile. At the same time, following \Cref{thm: pnkt and nu}, one also obtains an alternative proof of the limiting profile theorem of Nestoridi and Olesker-Taylor \cite[Theorem B]{nestoridi2022limit}.

\subsection{What's next?}

Our results for the repeated-card setting suggest several directions for further investigation. In the present paper, we only treat the regime $l=o(\log n)$, that is, when the multiplicity parameter tends to infinity very slowly. The main reason is technical: the Jain--Sawhney approach underlying \Cref{thm: pnkt and nu} only gives effective control in a relatively narrow time interval around $\frac{n}{k}\log n$. Nevertheless, the conclusion of \Cref{thm: window growing m} may well remain valid for substantially larger values of $l$.

A basic example is the case $k=2$, $m=2$, and $l=n/2$. In this regime, the model is exactly the simple exclusion process on the complete graph with $n$ vertices and $n/2$ particles. For this chain, Lacoin-Leblond \cite[Theorem 1.1]{MR2869447} showed that the mixing time is of order
\[
\frac{n}{2}\log \min\!\left(\frac{n}{2},\sqrt n\right)=\frac{n}{4}\log n,
\]
and that the cutoff window has order $n$. This is fully consistent with the prediction of \Cref{thm: window growing m}, which in this case gives the cutoff location
\[
\frac{n}{2}\left(\log n-\frac12\log \frac n2\right)
=
\frac{n}{4}\log n + O(n),
\]
together with a window of order $n$. The above observation motivates us to formulate the following conjecture.

\begin{conj}
Let $n=ml=\omega(1)$, where we require $m,l\ge 2$. Then the random transposition shuffle on repeated cards exhibits cutoff at time
$$\frac{n}{2}(\log n-\frac 12\log l)$$
with window of order $n$.
\end{conj}

It should be reasonable to replace random transpositions (resp. $2$-cycle) in the above conjecture by random $k$-cycles, where $k$ is not too large. In this case, the cutoff time would be replaced by $\frac{n}{k}(\log n-\frac 12\log l)$ and the time window would be replaced by $\frac nk$.

Another natural direction is to replace the random $k$-cycle shuffle by more general conjugacy-invariant shuffles on the symmetric group. Instead of taking the driving measure to be supported on a single conjugacy class of $k$-cycles, one may consider random walks generated by an arbitrary conjugacy class, or more generally by a class function on $\S_n$. From the representation-theoretic point of view, this would require extending the character estimates used in the present paper beyond the $k$-cycle setting. In this direction, one may hope to draw on the recent sharp character bounds developed by Olesker-Taylor, Teyssier, and Th\'evenin \cite{olesker2025sharp}. It would be very interesting to understand whether the repeated-card phenomena exhibited in the present paper--notably, the shift in the cutoff location and the transition from a Poisson profile to a Gaussian profile persist in such a general setting.

\subsection{Notation}

Throughout the paper, $d_{TV}(\cdot,\cdot)$ refers to the total variation distance, $\Pois(\cdot)$ refers to the Poisson distribution, and $\Phi(\cdot)$ refers to the cumulative distribution function of the normal distribution $\mathcal{N}(0,1)$. We will write $[n]$ for the set $\{1,2,\ldots,n\}$, and $2^{[n]}$ to be the set consisting of all subsets of $[n]$. When $n=ml$, we write $T_{i,l}:=\{i,i+m,\ldots,i+(l-1)m\}$ for all $1\le i\le m$ (i.e., all repeated cards of a single type), so that $[n]=\bigsqcup_{i=1}^{m}T_{i,l}$. We use the symbol $|\cdot|$ for the order of a finite set or the absolute value of a complex number.

We write \(f=O(g)\) to mean that \(|f| \leq C|g|\) for some absolute constant \(C>0\), and we write \(f=o(g)\) to mean that \(f/g \to 0\). Similarly, we write \(f=\omega(g)\) to mean that \(f/g \to +\infty\). Here all asymptotic notation is understood in absolute value. Moreover, all such asymptotic statements are taken along the given sequence of positive integer inputs. In particular, they are not meant uniformly over all such inputs, but only with respect to the particular sequence under consideration.

\subsection{Outline of the paper}

In \Cref{sec: Approximating the shuffling with an explicitly tractable measure}, we prove \Cref{thm: pnkt and nu}. Based on the approximation provided by \Cref{thm: pnkt and nu}, we prove \Cref{thm: window growing m} and \Cref{thm: window fixed m}  in \Cref{sec: From tractable measure to total variation distance}.

\section{Approximating the shuffling with an explicitly tractable measure}\label{sec: Approximating the shuffling with an explicitly tractable measure}

The purpose of this section is to prove the following theorem, which gives a detailed formulation of \Cref{thm: pnkt and nu}.

\begin{thm}\label{thm: restate of pnkt and nu}
Let $(n_N)_{N\ge 1},(k_N)_{N\ge 1}$, and $(t_N)_{N\ge 1}$ be sequences of positive integers, such that $\lim_{N\rightarrow\infty}n_N=\infty$. For all $N\ge 1$, let $t'_N:=t_N-\frac{n_N\log n_N}{k_N}$. Also, let $P_{n_N,k_N}^{*t_N}$ be defined as in \Cref{defi: Pnk and its convolution}, and $\nu_{n_N,k_N}^{t_N}$ be defined as in \Cref{defi: construction of nu}. 
\begin{enumerate}
\item Suppose that 
\begin{enumerate}
\item $k_N\ge 2$ for all $N\ge 1$, and $\lim_{N\rightarrow\infty}\frac{k_N\log n_N}{n_N}=0$.
\item $t_N'=cn_N/k_N$, where $c\in\R$ is an absolute constant.
\end{enumerate}
Then, we have
$$\lim_{N\rightarrow\infty}d_{TV}(P_{n_N,k_N}^{*t_N},\nu_{n_N,k_N}^{t_N})=0.$$ 
\label{item: dTV for large k}
\item Suppose that 
\begin{enumerate}
\item $k_N\ge 2$ for all $N\ge 1$, and $\limsup_{N\rightarrow\infty}\log_{n_N}k_N\le c_0$, where $c_0\in[0,1)$ is an absolute constant.
\item $\lim_{N\rightarrow\infty}\log\log n_N-|2k_Nt_N'/n_N|=+\infty$.
\end{enumerate}
Then, we have
$$\limsup_{N\rightarrow\infty}\log_{n_N}d_{TV}(P_{n_N,k_N}^{*t_N},\nu_{n_N,k_N}^{t_N})\le c_0-1.$$
\label{item: dTV for wide t}
\end{enumerate}
\end{thm}

For the rest of the section, for notational ease, we drop the subscripts and simply write $n,k$ for $n_N,k_N$. Now, let us recall necessary background from representation theory.

\subsection{Nonabelian Fourier transform and related notations}

Given a finite group $G$, we let $\widehat G$ denote the set of irreducible representations. We define convolution $f_1,f_2:G\rightarrow\C$ to be
$$(f_1*f_2)(z):=\sum_{x\in G}f_1(x)f_2(x^{-1}z).$$
For $\rho\in\widehat G$, let $d_\rho$ denote the dimension of $\rho$. The nonabelian Fourier transform for a function $f:G\rightarrow\C$ is the map $\widehat f$ given by
$$\widehat f(\rho):=\sum_{g\in G}f(g)\rho(g)\in\Mat_{d_\rho}(\C),\quad\forall \rho\in\widehat G.$$
Then, for all $f_1,f_2:G\rightarrow\C$ and $\rho\in\widehat G$, we have
\begin{equation}\label{eq: convolution to multiplication}
\widehat{f_1*f_2}(\rho)=\widehat f_1(\rho)\cdot\widehat f_2(\rho).
\end{equation}
In this paper, we always take $G=\mathfrak{S}_n$. Following the standard construction of Specht modules, we can identify elements in $\widehat{\mathfrak{S}}_n$ with partitions of $n$. From now on, we use $\l\vdash n$ to denote that $\l$ is a partition of $n$. Given a partition $\l=(\l_1,\ldots,\l_j)$ with $\l_1\ge\cdots\ge\l_j$, we let $\l'$ denote the conjugate partition and $\l^*=(\l_2,\ldots,\l_j)$ be the partition of $n-\l_1$ obtained by truncating $\l$. Denote by
$l(\lambda)$ the number of parts of $\lambda$, and
\[
\lambda^\bullet:=\bigl((\lambda')^*\bigr)'\,,
\]
which is obtained from $\lambda$ by deleting its first column. We will also apply the notation introduced in the following definition, which follows from \cite{nestoridi2022limit} and \cite{hough2016random}.

\begin{defi}
Suppose $\l\vdash n$ has diagonal length $m$ (see the top of \cite[page 454]{hough2016random}). The \emph{Frobenius notation} is the expression
$$\l=(a_1,\ldots,a_m|b_1,\ldots,b_m),$$
where 
$$a_i:=\lambda_i-i+\frac{1}{2},b_i:=\lambda_i'-i+\frac{1}{2},\quad\forall 1\le i\le m.$$
\end{defi}

Let
\begin{equation}\label{eq: long first row and column partitions}
\mathcal L \ :=\ \Bigl\{\lambda\vdash n:\ \lambda_1\ge n-\log n\ \ \text{or}\ \ l(\l)\ge n-\log n\Bigr\}
\end{equation}
be the set of partitions with long first row or long first column. It is clear that $\mathcal L=\mathcal L_1\bigsqcup \mathcal L_2$, where
$$\mathcal L_1:=\ \Bigl\{\lambda\vdash n:\ \lambda_1\ge n-\log n\Bigr\},\quad\mathcal L_2:=\ \Bigl\{\lambda\vdash n:\ l(\l)\ge n-\log n\Bigr\}.$$
Moreover, under the conjugation map, the partitions in $\mathcal{L}_1$ have one-to-one correspondence with the partitions in $\mathcal{L}_2$. For the rest of this section, we  usually use the letter $r$ for $n-\l_1$.

Probability measures on $\mathfrak{S}_n$ will be viewed as real-valued functions. In our applications, we will furthermore restrict attention to functions $f$ which are class functions, i.e., functions that are constants on conjugacy classes. 
\begin{lemma}[Schur's lemma]
Suppose $f:\mathfrak{S}_n\rightarrow \C$ is a class function. Then, we have
$$\widehat f(\lambda)=\frac{\sum_{g\in G}f(g)\chi_\lambda(g)}{d_\lambda}\Id_{d_\lambda},$$
where $\chi_\lambda$ is the character corresponding to $\lambda\vdash n$.
\end{lemma}

In particular, let $P_{n,k}$ be the probability measure over $\mathfrak{S}_n$ as in \eqref{eq: measure of k-cycle}. Then, its Fourier transform is given by
\begin{equation}\label{eq: fourier transform of Pn}
\widehat P_{n,k}(\lambda)=\frac{\chi_\lambda(\tau_k)}{d_\lambda}\Id_{d_\lambda},\quad \lambda\vdash n.
\end{equation}
where $\tau_k$ is an arbitrary $k$-cycle. Following \cite{nestoridi2022limit}, we will also use the abbreviation $s_\lambda(k)$ for $\frac{\chi_\lambda(\tau_k)}{d_\lambda}$. In particular, by the hook length formula, we have $d_\l=d_{\l'}$. Furthermore, denote by $S^{\lambda'}$ and $S^\lambda$ the Specht module of types $\lambda'$ and $\lambda$, and let $\sgn_n$ be the sign representation. Then, the natural isomorphism $S^{\lambda'}\cong S^\lambda\otimes\sgn_n$ leads to the fact that $\chi_{\l'}(k)=\chi_\l(k)\cdot(-1)^{k-1}$, and therefore
$$s_{\l'}(k)=s_\l(k)\cdot(-1)^{k-1}.$$


\subsection{Character estimates for $P_{n,k}^{*t}$}

Following \eqref{eq: convolution to multiplication} and \eqref{eq: fourier transform of Pn}, it is clear that
$$\widehat {P_{n,k}^{*t}}(\l)=s_\l(k)^t\Id_{d_\l},\quad\l\vdash n.$$
In the following, we provide quantitative bounds that will be useful later. 

\begin{lemma}[{\cite[Theorem 5(a)]{hough2016random}}]\label{lem: MT and ET}
Let $0<\epsilon<\frac{1}{2}$, let $r=n-\l_1$ and suppose that $r+k+1<(\frac12-\epsilon)n$. Then
\begin{multline}\label{eq: MT and ET}
s_\l(k)=\frac{(n-r-1)^{\underline{k}}}{n^{\underline{k}}}\prod_{i=2}^m\left(1-\frac{k}{n-(1+r+\lambda_i-i)}\right)\prod_{i=1}^m\left(1-\frac{k}{n-(r-\lambda_i'+i)}\right)^{-1}\\
+O_\epsilon\left(\exp\left(k\left[\log\frac{(1+\epsilon)(k+1+r)}{n-k}+O_\epsilon\left(r^{-\frac{1}{2}}\right)\right]\right)\right).
\end{multline}
Here, the superscript $\underline{k}$ is the falling factorial. Moreover, when $r<k$, the error term on the second line is actually zero.
\end{lemma}

Following \cite{hough2016random}, we will abbreviate $MT$ (resp. $ET$) for the main (resp. error) term in \eqref{eq: MT and ET}. As we shall see, when we use the above bound, the main term is usually close to $1$, and the error term is usually close to zero.

\begin{lemma}[{\cite[Lemma 13]{hough2016random}}]\label{lem: lower bound for MT}
Suppose that $k+r+1<n/2$, where $r:=n-\l_1$. Then, we have $MT\le \exp(-rk/n)$.
\end{lemma}

\begin{prop}\label{prop: small r}
Suppose $r:=n-\l_1\in[1, \log n]$, and $2\le k\le o(n/\log n)$. Moreover, suppose that $t':=t-\frac{n\log n}{k}$ satisfies $|t'|\le \frac{n(\log\log n-\omega(1))}{2k}$. Then, we have
$$s_\l(k)^t=\exp\left(-\frac{rkt}{n}-\frac{r(k+r)\log n}{2n}(1+o(1)+O(1/k))\right).$$
\end{prop}

\begin{proof}
We first estimate the main term in \eqref{eq: MT and ET}. 
Write
$$P_0:=\frac{(n-r-1)^{\underline{k}}}{n^{\underline{k}}},P_1:=\prod_{i=2}^m\left(1-\frac{k}{n-(1+r+\lambda_i-i)}\right),P_2:=\prod_{i=1}^m\left(1-\frac{k}{n-(r-\lambda_i'+i)}\right)^{-1}.$$
We have
\begin{align}
\label{eq: estimate of P0_small r}
\begin{split}
P_0&=\frac{n-r-k}{n-r}\cdot\frac{(n-r)\cdots(n-r-k+1)}{n(n-1)\cdots(n-k+1)}\\
&=(1-k/n)\left(1+\frac{O(kr)}{n^2}\right)\prod_{i=0}^{k-1}\left((1-r/n)\left(1-\frac{(1+o(1))i}{2n^2}\right)\right)\\
&=(1-r/n)^k(1-k/n)\left(1-\frac{(1+o(1))k^2r+O(kr)}{2n^2}\right).
\end{split}
\end{align}
Following the definition of the Frobenius notation, it is clear that $m(m-1)\le r$ and $\sum_{i=2}^m \l_i\le r$. Therefore, we have
\begin{align}\label{eq: estimate of P1_small r}
\begin{split}
P_1&=\prod_{i=2}^m\left((1-k/n)\cdot\left(1-\frac{k(1+r+\lambda_i-i)}{n^2}-\frac{k}{n^{3-o(1)}}\right)\right)\\
&=(1-k/n)^{m-1}\cdot\left(1-\frac{k\sum_{i=2}^m(1+r+\lambda_i-i)}{n^2}-\frac{k}{n^{3-o(1)}}\right)\\
&=(1-k/n)^{m-1}\cdot\left(1-\frac{k(m+O(1))r}{n^2}\right).
\end{split}
\end{align}
Notice that $\sum_{i=1}^m\lambda_i'\le r+m$, we also have
\begin{align}\label{eq: estimate of P2_small r}
\begin{split}
P_2&=\prod_{i=1}^m\left((1-k/n)^{-1}\cdot\left(1+\frac{k(r-\lambda_i'+i)}{n^2}+\frac{k}{n^{3-o(1)}}\right)\right)\\
&=(1-k/n)^{-m}\cdot\left(1+\frac{k\sum_{i=1}^m(r-\lambda_i'+i)}{n^2}+\frac{k}{n^{3-o(1)}}\right)\\
&=(1-k/n)^{-m}\cdot\left(1+\frac{k(m+O(1))r}{n^2}\right).
\end{split}
\end{align}
Combining the estimates in \eqref{eq: estimate of P0_small r}, \eqref{eq: estimate of P1_small r} and \eqref{eq: estimate of P2_small r}, we have
\begin{align}\label{eq: estimate of MT_small r}
\begin{split}
MT&=P_0P_1P_2\\
&=(1-r/n)^k\left(1-\frac{(1+o(1))k^2r+O(kr)}{2n^2}\right)\\
&=\exp\left(-\frac{rk}{n}-\frac{(1+o(1))(k^2r+kr^2)+O(kr)}{2n^2}\right).
\end{split}
\end{align}
We now deal with the error term in \eqref{eq: MT and ET}. When $k>r$, the error term is zero. If $k\le r$, then by taking $\epsilon=0.1$ we have 
\begin{equation}\label{eq: estimate of ET_small r}
ET=O(\exp(k[(-1+o(1))\log n]))=O\left(n^{(-1+o(1))k}\right)=O\left(\frac{1}{n^{3-o(1)}}\right)
\end{equation}
when $k\ge 3$, and
\begin{equation}\label{eq: estimate of ET_small r_k=2}
ET=O(\exp(2(\log r-\log n+O(1)))=O\left(\frac{r^2}{n^2}\right)
\end{equation}
when $k=2$. Therefore, by \Cref{lem: MT and ET} and the estimates in \eqref{eq: estimate of MT_small r}, \eqref{eq: estimate of ET_small r} and \eqref{eq: estimate of ET_small r_k=2}, we have
$$s_\lambda(k)=MT+ET=\exp\left(-\frac{rk}{n}-\frac{(1+o(1))(k^2r+kr^2)+O(kr+r^2)}{2n^2}\right),$$
and therefore
\begin{align}
\begin{split}
s_\l(k)^{t}&=\exp\left(-\frac{rkt}{n}-t\cdot\frac{(1+o(1))(k^2r+kr^2)+O(kr+r^2)}{2n^2}\right)\\
&=\exp\left(-\frac{rkt}{n}-\frac{r(k+r)\log n}{2n}(1+o(1)+O(1/k))\right)\\
\end{split}
\end{align}
since $t=(1+o(1))n\log n/k$.
\end{proof}

The following result is a technical adjustment of \Cref{prop: small r}. However, in this case, only an upper bound is needed.

\begin{prop}\label{prop: middle r and ml r and large k}
Let $r:=n-\l_1$. Suppose that one of the following holds:
\begin{enumerate}
\item $r\in[\log n,n^{5/6}]$, and $2\le k\le o(n/\log n)$. \label{item: middle r}
\item $r\in[n^{5/6},0.499 n]$, and $6\log n\le k\le o(n/\log n)$. \label{item: ml r and large k}
\end{enumerate}
Moreover, suppose that $t':=t-\frac{n\log n}{k}$ satisfies $|t'|\le \frac{n(\log\log n-\omega(1))}{2k}$. Then, we have 
$$|s_\l(k)|^{2t}\le \exp\left(-2r\log n+r\log r-\omega(r)\right).$$
\end{prop}

\begin{proof}
Following \Cref{lem: lower bound for MT}, we have
\begin{equation}\label{eq: estimate of MT_middle r}
MT\le \exp(-rk/n).
\end{equation}

Now, let us bound the error term by taking $\epsilon=10^{-10}$. For case \eqref{item: middle r}, when $k>r$, the error term is zero. If $k\le r$, we have that
\begin{equation}
ET=O(\exp(k[(-1+o(1))\log\frac nr]))=O\left(\left(\frac rn\right)^{(1-o(1))k}\right),
\end{equation}
and therefore
$$|s_\l(k)|=MT+ET\le\exp(-rk/n)+O\left(\left(\frac rn\right)^{(1-o(1))k}\right)=\exp\left(-\frac{rk}{n}+O\left(\left(\frac rn\right)^{(1-o(1))k}\right)\right).$$
Notice that 
$$2kt/n\ge2\log n-\log\log n+\omega(1),\quad\log r\ge \log \log n,$$ 
we have $-2rkt/n\le -2r\log n+\frac23 r\log r-\omega(r)$. Moreover, we have $t\cdot\left(\frac rn\right)^{(1-o(1))k}\le n\log n\cdot\left(\frac rn\right)^{2-o(1)}=o(r)$, and thus
$$|s_\l(k)|^{2t}\le\exp\left(-\frac{2rkt}{n}+O\left(t\cdot\left(\frac rn\right)^{(1-o(1))k}\right)\right)\le\exp\left(-2r\log n+r\log r-\omega(r)\right).$$

For case \eqref{item: ml r and large k}, when $k>r$, the error term is zero. If $k\le r$, we have 
\begin{equation}
ET=O\left(\exp\left(k\cdot\log\frac 12\right)\right)=O(\exp(-k\log 2)),
\end{equation}
and therefore 
\begin{align}
\begin{split}
|s_\l(k)|&= MT+ET\\
&\le\exp(-rk/n)+O(\exp(-k\log 2))\\
&=\exp\left(-\frac{rk}{n}\right)(1+O(\exp(-k(\log 2-0.5)))\\
&=\exp\left(-\frac{rk}{n}+O(n^{-1.15})\right).
\end{split}
\end{align}
Here, in the last line, we are using the facts that $k\ge 6\log n$ and $6(\log 2-0.5)\ge 1.15$. For the same reason as in case \eqref{item: middle r}, we have $-2rkt/n\le -2r\log n+r\log r-\omega(r)$, and therefore
$$|s_\l(k)|^{2t}\le\exp\left(-\frac{2rkt}{n}+O\left(tn^{-1.15}\right)\right)\le\exp\left(-2r\log n+r\log r-\omega(r)\right).$$
Our bounds for case \eqref{item: middle r} and \eqref{item: ml r and large k} together give the proof.
\end{proof}

\begin{prop}\label{prop: ml r small k}
Suppose that $r:=n-\l_1\in[n^{5/6},0.499n]$. Furthermore, suppose $2\le k\le 6\log n$. Moreover, suppose that $t':=t-\frac{n\log n}{k}$ satisfies $|t'|\le \frac{n(\log\log n-\omega(1))}{2k}$. Then, for sufficiently large $n$, we have
$$|s_\l(k)|^{2t}\le\exp(-2r\log n+(2/3+o(1))r\log r).$$
\end{prop}

When the parameters $r,k$ lie in the range of \Cref{prop: ml r small k}, the bound in \Cref{lem: MT and ET} no longer works. Therefore, in order to prove \Cref{prop: ml r small k}, we need some new estimates, which we list below.


\begin{lemma}[{\cite[Lemma 14]{hough2016random}}]\label{lem: lem14 from Hough}
Suppose that $r:=n-\l_1\in[n^{5/6},0.499n]$. Furthermore, suppose $2\le k\le 6\log n$. Then, we have the bound
$$|s_\l(k)|\le\left(1+O\left(\frac{\log n}{n^{1/4}}\right)\right)\left[\sum_{a_i>kn^{1/2}}\frac{a_i^k}{n^k}+\sum_{b_i>kn^{1/2}}\frac{b_i^k}{n^k}\right]+O\left(\frac{\exp(-k)\log n}{n^{1/4}}\right).$$
\end{lemma}

\begin{lemma}[{\cite[Lemma 15]{hough2016random}}]\label{lem: lem15 from Hough}
Let $k\ge 2$, and let $r:=n-\l_1\in[n^{5/6},0.499n]$. Assuming $l(\l)\le \l_1$. Then, we have
$$\sum_{a_i>kn^{1/2}}\frac{a_i^k}{n^k}+\sum_{b_i>kn^{1/2}}\frac{b_i^k}{n^k}\le\exp\left(-\frac{rk}{2n}\right).$$
Moreover, there exists a constant $c_0>0$ such that the following holds. Suppose $r>c_0n$, then
$$\sum_{a_i>kn^{1/2}}\frac{a_i^k}{n^k}+\sum_{b_i>kn^{1/2}}\frac{b_i^k}{n^k}\le\exp\left(-\frac{rk}{n}+\frac{kr^2}{n^2}\right).$$
\end{lemma}

\begin{proof}[Proof of \Cref{prop: ml r small k}]
Let $c_0$ be the same as in \Cref{lem: lem15 from Hough}. On the one hand, if $r\le \min\{c_0,1/5\}n$, we have
$$\frac{r^2}{n^2}\le\frac{r}{5n}\le\frac{r\log (n^{5/6})}{4n\log n}\le\frac{r\log r}{4n\log n}.$$
Therefore,
$$\log\left(\sum_{a_i>kn^{1/2}}\frac{a_i^k}{n^k}+\sum_{b_i>kn^{1/2}}\frac{b_i^k}{n^k}\right)\le\exp\left(-\frac{rk}{n}+\frac{kr^2}{n^2}\right)\le\exp\left(-\frac{rk}{n}+\frac{kr\log r}{4n\log n}\right).$$
In this case, applying \Cref{lem: lem14 from Hough}, we have
\begin{align}
\begin{split}
|s_\l(k)|&\le\left(1+O\left(\frac{\log n}{n^{1/4}}\right)\right)\exp\left(-\frac{rk}{n}+\frac{kr\log r}{4n\log n}\right)+O\left(\frac{\exp(-k)\log n}{n^{1/4}}\right)\\
&\le\exp\left(-\frac{rk}{n}+\frac{rk\log r}{3n\log n}\right).
\end{split}
\end{align}
Here, the last line follows because $r\ge n^{5/6}$ so that $\frac{r\log r}{n\log n}\ge n^{-1/6+o(1)}$. On the other hand, if $r\ge \min\{c_0,1/5\}n$, we already have $-\frac{rk}{2n}\le-\frac{rk}{n}+\frac{rk\log r}{4n\log n}$ when $n$ is sufficiently large, and therefore the inequality $|s_\l(k)|\le\exp\left(-\frac{rk}{n}+\frac{rk\log r}{3n\log n}\right)$ follows the same way as the above. Thus, when $n$ is sufficiently large, the inequality
$$|s_\l(k)|^{2t}\le\exp\left(-\frac{2rkt}{n}+\frac{2r\log rkt}{3n\log n}\right)=\exp(-2r\log n+(2/3+o(1))r\log r)$$
always holds.
\end{proof}

The following result is quoted from \cite[Theorem 3.5]{nestoridi2022limit}. 

\begin{prop}\label{prop: short first row and large k}
Suppose that $6\log n<k=o(n/\log n)$, and let $\theta:=0.68$, so that $\exp(-\theta)>0.506$. Consider $\lambda\vdash n$ with $b_1\le a_1<e^{-\theta}n$. Then, we have
$$|s_\lambda(k)|\le \exp(-0.51k).$$
\end{prop}

\begin{proof}
Choose $\theta':=0.677$, so that $\theta'-\frac{1}{6}>0.51$. Now, inspection of the proof of \cite[Theorem 5(b)]{hough2016random} gives the bound
$$|s_\l(k)|\le\exp(k(-\theta'+1/6+o(1)))\le\exp(-0.51k).$$
\end{proof}

\begin{rmk}
In the original version, \cite[Theorem 5(b)]{hough2016random} provides the slightly weaker bound
\[
|s_\lambda(k)| \le \exp(-0.5\,k).
\]
For our technical purposes this estimate is not sufficient (see the proof of \Cref{thm: removal}),
since we need a strict constant improvement in the exponent (the improvement may be
arbitrarily small, but must be positive). Fortunately, by following the proof of
\cite[Theorem~5(b)]{hough2016random} and optimizing the numerical constants, one can obtain such
an improvement with no additional ideas.
\end{rmk}

\begin{prop}[{\cite[Lemma 3.6]{nestoridi2022limit}}]\label{prop: short first row}
The following statement holds when $n$ is sufficiently large. Assume that $2\le k\le 6 \log n$. Let $\lambda\vdash n$ such that $l(\lambda)\le\lambda_1$ and $r:=n-\lambda_1\in[\frac{1}{3}n,n]$. Then, we have
$$|s_\lambda(k)|\le \exp\left(-\frac{3rk}{5n}\right).$$
\end{prop}

\begin{lemma}[{\cite[Lemma 3.3]{jain2024hitting}}]\label{lem: dimension bound}
Let $\l\vdash n$, and $r:=n-\l_1$. Then, we have
\begin{enumerate}
\item $d_\lambda\le\binom{n}{r}\cdot \sqrt{r!}\le n^r/\sqrt{r!}.$
\item When $r\le \log n$, we have
$$d_\l=\binom{n}{r}\cdot d_{\l^*}\left(1-\frac{r}{n}+O(n^{-2+o(1)})\right).$$
\end{enumerate}
\end{lemma}

The following bound generalizes \cite[Lemma 3.4]{jain2024hitting}. It shows that the sum $\sum_{\lambda\vdash n}d_\l^2|s_\l(k)|^{2t}$ is mainly contributed by partitions in $\mathcal{L}$. 

\begin{thm}\label{thm: removal}
Suppose that $2\le k\le o(n/\log n)$. Moreover, suppose that $t':=t-\frac{n\log n}{k}$ satisfies $|t'|\le \frac{n(\log\log n-\omega(1))}{2k}$. Then, we have
$$\sum_{\substack{\lambda\vdash n\\\lambda\notin \mc{L}}}d_\l^2|s_\l(k)|^{2t}=n^{-\omega(1)}.$$
\end{thm}

\begin{proof}

\medskip
\noindent\textbf{Step 1: reduce to $l(\lambda)\le \lambda_1$.}
Partition conjugation preserves dimensions ($d_{\lambda'}=d_\lambda$), and since the step measure is a
class function supported on $k$-cycles we have $|s_{\lambda'}(k)|=|s_\lambda(k)|$.
Moreover, $\lambda\notin\mc L$ implies $\lambda'\notin\mc L$ by symmetry of \eqref{eq: long first row and column partitions}.
Hence it suffices to bound
\begin{equation}\label{eq:reduce-half}
\sum_{\substack{\lambda\vdash n\\ \lambda\notin\mc L\\ l(\lambda)\le \lambda_1}}
d_\lambda^2\,|s_\lambda(k)|^{2t},
\end{equation}
and the full sum is at most twice~\eqref{eq:reduce-half}.

\medskip
\noindent\textbf{Step 2: split by the size of $r=n-\lambda_1$.}
Fix $\lambda$ with $l(\lambda)\le \lambda_1$ and $\lambda\notin\mc L$, and write $r:=n-\lambda_1$.
Then, we have $r\ge \log n$. Denote by $p(r)$ the partition number of $r$. Recall that the Hardy-Ramanujan formula for partitions gives the asymptotic estimate $p(r)\lesssim\exp(\pi\sqrt{2r/3})$. Moreover, we have the estimate $d_\l^2\le \frac{n^{2r}}{r!}$ given in \Cref{lem: dimension bound}. We will use these bounds in the following three regimes.

\medskip
\noindent\textbf{Regime I: $r\in[\log n,n^{5/6}] $.}
Applying \Cref{prop: middle r and ml r and large k}, we have
$$|s_\l(k)|^{2t}\le \exp\left(-2r\log n+r\log r-\omega(r)\right),$$
and therefore
\[
\sum_{\l_1=n-r}d_\lambda^2\,|s_\lambda(k)|^{2t}
 \le \frac{n^{2r}}{r!}\cdot p(r)\cdot|s_\l(k)|^{2t}\le\exp(-\omega(r))
 = n^{-\omega(1)}.
\]
Summing over $r$ gives $n^{-\omega(1)}$.
\medskip

\noindent\textbf{Regime II: $r\in[n^{5/6}, 0.499n]$.}
We split further according to $k$.

\begin{enumerate}
\item Suppose $2\le k\le 6\log n$. Applying \Cref{prop: ml r small k}, we have
\[
|s_\lambda(k)|^{2t}
 \le \exp(-2r\log n+(2/3+o(1))r\log r),
\]
and therefore 
\[
\sum_{\l_1=n-r}d_\lambda^2\,|s_\lambda(k)|^{2t}
 \le \frac{n^{2r}}{r!}\cdot p(r)\cdot|s_\l(k)|^{2t}\le\exp((1/3-o(1))r\log r)
 = n^{-\omega(1)}.
\]
Summing over $r$ gives $n^{-\omega(1)}$.
\item Suppose $6\log n<k\le o(n/\log n)$. Applying \Cref{prop: middle r and ml r and large k}, we have
$$|s_\l(k)|^{2t}\le \exp\left(-2r\log n+r\log r-\omega(r)\right),$$
and therefore
\[
\sum_{\l_1=n-r}d_\lambda^2\,|s_\lambda(k)|^{2t}
 \le \frac{n^{2r}}{r!}\cdot p(r)\cdot|s_\l(k)|^{2t}\le\exp(-\omega(r))
 = n^{-\omega(1)}.
\]
Summing over $r$ gives $n^{-\omega(1)}$.
\end{enumerate}

\medskip
\noindent\textbf{Regime III: $r\in[0.499n, n]$.}
We split further according to $k$. 
\begin{enumerate}
\item Suppose $2\le k\le 6\log n$. Applying \Cref{prop: short first row}, we have
\[
|s_\lambda(k)|^{2t}
 \le \exp\left(-\frac{6rkt}{5n}\right)
\le\exp(-1.1r\log n)
\]
when $n$ is sufficiently large, and therefore 
\[
\sum_{\l_1=n-r}d_\lambda^2\,|s_\lambda(k)|^{2t}
 \le \frac{n^{2r}}{r!}\cdot p(r)\cdot\exp(-1.1r\log n)=\exp(0.1r\log r+O(r))
 = n^{-\omega(1)}.
\]
Summing over $r$ gives $n^{-\omega(1)}$.
\item Suppose $6\log n<k\le o(n/\log n)$. Applying \Cref{prop: short first row and large k}, we have
\[
|s_\lambda(k)|^{2t}
 \le \exp(-1.02 kt)
 \le \exp(-1.01n\log n)
\]
when $n$ is sufficiently large, and therefore
\[
\sum_{\l_1=n-r} d_\lambda^2\,|s_\lambda(k)|^{2t}
\ \le\ \Big(\sum_{\lambda\vdash n} d_\lambda^2\Big)\exp(-1.01n\log n) = n!\exp(-1.01n\log n)
 = n^{-\omega(1)}.
\]
Summing over $r$ gives $n^{-\omega(1)}$.
\end{enumerate}
\medskip

Combining the three regimes discussed above, we complete the proof.
\end{proof}

\subsection{Representation-theoretic preliminaries for the Fourier transform of $\nu_{n,k}^t$}

Fix integers $n\ge 1$ and $0\le M\le n/3$. In this case, for $\l\vdash n$, at most one of the following holds:
\begin{enumerate}
\item $\l_1\ge n-M$.
\item $l(\l)\ge n-M$.
\end{enumerate}

Define a probability measure $\xi_M$ (resp. $\xi_M^c$) on $\mathfrak{A}_n$ as follows:
choose a subset $S\subset[n]$ uniformly among all subsets of size $M$, then sample a permutation
$\pi$ uniformly from $\A_{[n]\setminus S}$ (resp. $\A_{[n]\setminus S}^c$ ), and extend it to an element of $\mathfrak{A}_n$ (resp. $\mathfrak{A}_n^c$)
by fixing every point of $S$. The following lemma gives an expression for the Fourier transform of $\xi_M$ and $\xi_M^c$.

\begin{thm}[Fourier transform of $\xi_M$ and $\xi_M^c$]
\label{thm :zetaM_fourier_An}
We have
\[
\widehat{\xi_M}(\lambda)
=
\frac{K_{\lambda,\mu}+K_{\lambda',\mu}}{d_\lambda}\,\Id_{d_\lambda}, \quad\widehat{\xi_M^c}(\lambda)
=
\frac{K_{\lambda,\mu}-K_{\lambda',\mu}}{d_\lambda}\,\Id_{d_\lambda},
\]
where
\[
\mu=\mu_M:=(n-M,1,1,\dots,1)\vdash n,
\]
and $K_{\alpha,\mu}$ denotes the Kostka number.

Moreover, if $\l_1<n-M$ and $l(\l)<n-M$, then $\widehat{\xi_M}(\lambda)=\widehat{\xi_M^c}(\lambda)=0$. If $\l_1\ge n-M$, we have
$$\widehat{\xi_M}(\lambda)=\binom{M}{\,n-\lambda_1\,}\frac{
d_{\lambda^*}}{d_\lambda}\,\Id_{d_\lambda},\quad\widehat{\xi_M^c}(\lambda)=\binom{M}{\,n-\lambda_1\,}\frac{
d_{\lambda^*}}{d_\lambda}\,\Id_{d_\lambda}.$$
If $l(\l)\ge n-M$, we have
$$\widehat{\xi_M}(\lambda)=\binom{M}{\,n-l(\lambda)\,}\frac{
d_{\lambda^\bullet}
}{d_\lambda}\,\Id_{d_\lambda},\quad\widehat{\xi_M^c}(\lambda)=-\binom{M}{\,n-l(\lambda)\,}\frac{
d_{\lambda^\bullet}
}{d_\lambda}\,\Id_{d_\lambda}.$$
\end{thm}

\begin{proof}

We will only prove the theorem for $\xi_M$, since the proof for $\xi_M^c$ is analogous.

\textbf{Step 1: rewrite as a subgroup average.}
Denote
\[
H:=\S_{\{1\}}\times \cdots \times \S_{\{M\}}\times \A_{\{M+1,\dots,n\}}.
\]
By conjugation-invariance of $\chi_\lambda$, we may write $\widehat{\xi_M}(\lambda)=a_\lambda\Id_{d_\lambda}$, where
$$
a_\lambda=\frac{1}{d_\lambda}\cdot\sum_{\sigma\in \mathfrak{S}_n}\xi_M(\sigma)\chi_\lambda(\sigma)=
\frac{1}{d_\lambda}\cdot\frac{1}{|\A_{n-M}|}\sum_{\sigma\in H}\chi_\lambda(\sigma).
$$
Let $\bbone_H$ denote the trivial character of $H$. By Frobenius reciprocity, we have
$$
\frac{1}{|\A_{n-M}|}\sum_{\sigma\in H}\chi_\lambda(\sigma)=\left\langle\bbone_H,\Res_H^{\S_n}\chi_\lambda\right\rangle_H=
\left\langle \Ind_H^{\mathfrak{S}_n}\bbone_H,\ \chi_\lambda\right\rangle_{\mathfrak{S}_n}.
$$
Therefore, we have 
\[
a_\lambda
=
\frac{1}{d_\lambda}\left\langle \Ind_H^{\mathfrak{S}_n}\bbone_H,\ \chi_\lambda\right\rangle_{\mathfrak{S}_n}.
\]
Thus, it remains to identify the coefficient of $\chi_\l$ when we write $\Ind_H^{\mathfrak{S}_n}\bbone_H$ as a linear sum of the form $\chi_\alpha$ with $\alpha\vdash n$, which is what we will do in the following.

\textbf{Step 2: split $\Ind_H^{\mathfrak{S}_n}\bbone_H$ into two induced pieces.}
Let
\[
\S_\mu:=\S_{\{1\}}\times\cdots\times \S_{\{M\}}\times \S_{\{M+1,\dots,n\}},
\]
the Young subgroup corresponding to $\mu=(n-M,1^M)$. Then $H\subset \S_\mu$ is a subgroup of index $2$, so the induced representation satisfies
\[
\Ind_H^{\S_\mu}\bbone_H \;\cong\; \bbone_{\S_\mu} \oplus \widetilde{\sgn},
\]
where $\widetilde{\sgn}$ is the inflation of the sign representation from the factor $\S_{\{M+1,\dots,n\}}$ to $\S_\mu$.
By transitivity of induction,
\begin{equation}\label{eq: summand of induction}
\Ind_H^{\mathfrak{S}_n}\bbone_H
\cong
\Ind_{\S_\mu}^{\mathfrak{S}_n}\bbone_{\S_\mu}
 \oplus
\Ind_{\S_\mu}^{\mathfrak{S}_n}\widetilde{\sgn}.
\end{equation}

\textbf{Step 3: identify multiplicities using Young's rule and tensoring by the sign representation.}
Recall that Young's rule gives the decomposition
\[
\Ind_{\S_\mu}^{\mathfrak{S}_n}\bbone_{\S_\mu}
=
\sum_{\alpha\vdash n} K_{\alpha,\mu}\,\chi_\alpha,
\]
so the multiplicity of $\chi_\lambda$ in the first summand of \eqref{eq: summand of induction} is $K_{\lambda,\mu}$.

Next, observe that $\widetilde{\sgn}=\Res_{\S_\mu}^{\S_n}\sgn_n$ (since the $\S_{\{i\}}$ factors are
trivial for all $1\le i\le M$), hence
$$
\Ind_{\S_\mu}^{\mathfrak{S}_n}\widetilde{\sgn}
=
\Ind_{\S_\mu}^{\mathfrak{S}_n}\left(\Res_{\S_\mu}^{\S_n}\sgn_n\right)
\cong
\left(\Ind_{\S_\mu}^{\mathfrak{S}_n}\bbone_{\S_\mu}\right)\otimes \sgn_n.
$$
Notice that tensoring an irreducible $\mathfrak{S}_n$-representation with $\sgn_n$ conjugates the partition, i.e.,
$\chi_\alpha\otimes \sgn_n=\chi_{\alpha'}$. Therefore,
\[
\Ind_{\S_\mu}^{\mathfrak{S}_n}\widetilde{\sgn}
=
\sum_{\alpha\vdash n} K_{\alpha,\mu}\,\chi_{\alpha'},
\]
so the multiplicity of $\chi_\lambda$ in the second summand of \eqref{eq: summand of induction} is $K_{\lambda',\mu}$.
Combining the two summands yields
\[
\left\langle \Ind_H^{\mathfrak{S}_n}\bbone_H,\ \chi_\lambda\right\rangle_{\mathfrak{S}_n}
=
K_{\lambda,\mu}+K_{\lambda',\mu},
\]
and hence
$$
\widehat{\xi_M}(\lambda)=\frac{K_{\lambda,\mu}+K_{\lambda',\mu}}{d_\lambda}\Id_{d_\lambda}.
$$

\textbf{Step 4: explicit form of the Kostka number.} As in the proof of \cite[Lemma 3.5]{jain2024hitting}, we have
$$
K_{\lambda,\mu}=\begin{cases}
0 & \l_1<n-M \\
d_{\lambda^*}\binom{M}{n-\lambda_1} & \l_1\ge n-M
\end{cases}.
$$
Similarly, we have the conjugation analogue given by
$$
K_{\lambda',\mu}=\begin{cases}
0 & l(\l)<n-M \\
\binom{M}{n-l(\lambda)}d_{\lambda^\bullet} & l(\l)\ge n-M
\end{cases}.
$$
Substituting into the Kostka-number formula gives the claimed explicit expression.
\end{proof}

\subsection{Proof of \Cref{thm: restate of pnkt and nu}}

We already have the necessary preparation to prove \Cref{thm: restate of pnkt and nu}. 

\begin{proof}[Proof of \Cref{thm: restate of pnkt and nu}]
For the rest of the proof, unless otherwise specified, the deductions work for both cases \eqref{item: dTV for large k} and \eqref{item: dTV for wide t}. In order to simplify computations, we consider the following truncated version $\mu_{n,k}^t$ of
$\nu_{n,k}^t$, which is defined in a similar fashion, except that the distribution of $M_t$ (i.e., the size of the random subset $S_t$) is given by
$$\mathbf{P}[M_t=x]=\mathbf{P}[\Pois(\gamma_t)=x]/\mathbf{P}[\Pois(\gamma_t)\le\log n],\quad x\le \log n.$$
Here, $\gamma_t:=\exp(-kt'/n)$ follows from \Cref{defi: construction of nu}. Due to our assumption on $t'$ for both cases \eqref{item: dTV for large k} and \eqref{item: dTV for wide t}, we have $\gamma_t=o(\sqrt{\log n})$, and therefore
$$\mathbf{P}[\Pois(\gamma_t)\ge \log n]\le \mathbf{P}[\Pois(\sqrt{\log n})\ge \log n]=n^{-\omega(1)}.$$
Therefore, it follows from the natural decoupling that $d_{TV}(\nu_{n,k}^t,\mu_{n,k}^t)=n^{-\omega(1)}$. Thus, it suffices to prove the same bound for $d_{TV}(P_{n,k}^{*t},\mu_{n,k}^t)$. Let $f_1$ be the probability density function of $P_{n,k}^{*t}$, and $f_2$ be the probability density function of $\mu_{n,k}^t$. Using the Cauchy–Schwarz inequality followed by the Plancherel formula, we have that (the dagger superscript refers to the conjugate transpose matrix)
\begin{align}
\begin{split}
4\cdot d_{TV}(P_{n,k}^{*t},\mu_{n,k}^t)^2&\le n!\sum_{\sigma\in\S_n}(f_1(\sigma)-f_2(\sigma))^2\\
&=\sum_{\l\vdash n}d_\l\Tr\left((\widehat f_1(\l)-\widehat f_2(\l))(\widehat f_1(\l)-\widehat f_2(\l))^\dag\right)\\
&=\sum_{\l\notin\L}d_\l^2|s_\l(k)|^{2t}+\sum_{\l\in\L}d_\l\Tr\left((\widehat f_1(\l)-\widehat f_2(\l))(\widehat f_1(\l)-\widehat f_2(\l))^\dag\right)\\
&=n^{-\omega(1)}+\sum_{\l\in\L}d_\l\Tr\left((\widehat f_1(\l)-\widehat f_2(\l))(\widehat f_1(\l)-\widehat f_2(\l))^\dag\right).\\
\end{split}
\end{align}
Here, the third line follows from \Cref{thm :zetaM_fourier_An} since $\widehat f_2(\l)=0$ when $\l\notin\L$, and the final line follows from \Cref{thm: removal}. Moreover, due to the symmetry under conjugation, the partitions in $\mathcal{L}_1$ and $\mathcal{L}_2$ contribute the same to the whole sum, and we only need to prove
\begin{equation*}
\sum_{\l\in\L_1}d_\l\Tr\left((\widehat f_1(\l)-\widehat f_2(\l))(\widehat f_1(\l)-\widehat f_2(\l))^\dag\right)=\begin{cases}
o(1) & \text{Case } \eqref{item: dTV for large k}\\
n^{2c_0-2-o(1)} & \text{Case } \eqref{item: dTV for wide t}
\end{cases}.
\end{equation*}
Now, suppose that $\l\in\mathcal{L}_1$, and denote $r:=n-\l_1$. On the one hand, by \Cref{prop: small r}, we have
\begin{equation}
\widehat f_1(\l)=s_\l(k)^t\Id_{d_\l}=\exp\left(-\frac{rkt}{n}-\frac{r(k+r)\log n}{2n}(1+o(1)+O(1/k))\right)\Id_{d_\l}.
\end{equation}
On the other hand, notice that
$$\mu_{n,k}^t\sim\begin{cases}
\sum_{l=0}^{\log n}\xi_l^c\cdot\frac{\mathbf{P}[\Pois(\gamma_t)=l]}{\mathbf{P}[\Pois(\gamma_t)\le\log n]} & t \text{ odd},k\text{ even}\\
\sum_{l=0}^{\log n}\xi_l\cdot\frac{\mathbf{P}[\Pois(\gamma_t)=l]}{\mathbf{P}[\Pois(\gamma_t)\le\log n]} & \text{else}
\end{cases},$$
where $\xi_l,\xi_l^c$ are the distributions in the statement of \Cref{thm :zetaM_fourier_An}. Therefore, it follows from \Cref{thm :zetaM_fourier_An} and \Cref{lem: dimension bound} that
\begin{align}
\begin{split}
\widehat f_2(\l)&=\mathbf{P}[\Pois(\gamma_t)\le\log n]^{-1}\cdot\sum_{l=0}^{\log n}\frac{d_{\lambda^*}\binom{l}{r}}{d_\lambda}\mathbf{P}[\Pois(\gamma_t)=l]\Id_{d_\l}\\
&=\left(1-\frac rn+O(n^{-2+o(1)})\right)\sum_{l=0}^{\log n}\binom{l}{r}\cdot r!n^{-r}\cdot\frac{\gamma_t^l e^{-\gamma_t}}{l !}\Id_{d_\l}\\
&=\left(1-\frac rn+O(n^{-2+o(1)})\right)\sum_{l=r}^{\log n}n^{-r}\cdot\frac{\gamma_t^l e^{-\gamma_t}}{(l-r)!}\Id_{d_\l}\\
&=\left(1-\frac rn+O(n^{-2+o(1)})\right)n^{-r}\cdot\gamma_t^r\cdot\mathbf{P}[\Pois(\gamma_t)\le \log n-r]\Id_{d_\l}\\
&=\left(1-\frac rn+O(n^{-2+o(1)})\right)\exp\left(-rkt/n\right)\mathbf{P}[\Pois(\gamma_t)\le\log n-r]\Id_{d_\l}.
\end{split}
\end{align}

We are close to done at this point. Suppose
$$\widehat f_1(\l)-\widehat f_2(\l)=\alpha_\l\Id_{d_\l},\quad \l\in\L_1,$$
so that 
$$\sum_{\l\in\L_1}d_\l\Tr\left((\widehat f_1(\l)-\widehat f_2(\l))(\widehat f_1(\l)-\widehat f_2(\l))^\dag\right)=\sum_{\l\in\L_1}d_{\l}^2|\alpha_{\l}|^2.$$
We break the above sum into two parts, depending on whether $r\ge\frac12\log n$ or not.
\begin{enumerate}
\item Suppose $r\ge \frac 12\log n$. In this case, it is clear that
$$\frac{r(k+r)\log n}{2n}(1+o(1)+O(1/k))\ge o(1)$$ for both cases \eqref{item: dTV for large k} and \eqref{item: dTV for wide t}. This implies
$$\widehat f_1(\l)\le\exp(-rkt/n)(1+o(1)).$$
Therefore, we can use the rough bound $\widehat f_1(\l)-\widehat f_2(\l)=\alpha_\l\Id_{d_\l}$ where $|\alpha_\l|\le 3\exp(-rkt/n)$, along with the dimension bound $d_\l^2\le n^{2r}/r!$ in \Cref{lem: dimension bound} to see that
\begin{align}
\begin{split}
\sum_{n-\l_1\in[\frac 12\log n,\log n]}d_\l^2|\alpha_\l|^2&\le9\sum_{r\in[\frac 12\log n,\log n]}\sum_{\l_1=n-r}\exp(-2rkt/n)\frac{n^{2r}}{r!}\\
&\le O(\log n)\cdot\max_{r\in[\frac 12\log n,\log n]}\frac{\exp(-2rkt'/n)}{r!}\exp\left(\pi\sqrt{2r/3}\right)\\
&= n^{-\omega(1)}.
\end{split}
\end{align}
Here, in the last inequality, we used Stirling's formula and the assumption $|2kt'/n|= \log\log n-\omega(1)$ for both cases \eqref{item: dTV for large k} and \eqref{item: dTV for wide t}.

\item Suppose $r<\frac 12\log n$. In this case, we have 
$$\mathbf{P}[\Pois(\gamma_t)\le \log n-r]\ge \mathbf{P}\left[\Pois(\sqrt {\log n})\le \frac 12\log n\right]\ge1-n^{-\omega(1)},$$ 
and therefore
$$\widehat f_2(\l)=\left(1-\frac rn+O(n^{-2+o(1)})\right)\exp\left(-rkt/n\right).$$
Also, by \Cref{lem: dimension bound}, we have
\begin{equation}\label{eq: estimate of sum of d_l^2}
\sum_{\l_1=n-r}d_\l^2\le \binom{n}{r}^2\sum_{\mu\vdash r}d_\mu^2=\binom{n}{r}^2\cdot r!\le\frac{n^{2r}}{r!}.
\end{equation}

For case \eqref{item: dTV for large k}, we split the interval $[1,\frac12 \log n)$ into two smaller parts. If $r<\min\{\sqrt{\frac{n}{k\log n}},\frac 12\log n\}$, we have
$$\frac{r(k+r)\log n}{2n}(1+o(1)+O(1/k))=o(1),$$
and therefore
\begin{align*}
\begin{split}
\alpha_\lambda&=\exp(-rkt/n)\left(\exp\left(\frac{-r(k+r)\log n}{2n}(1+o(1)+O(1/k))\right)-\left(1-\frac rn+O(n^{-2+o(1)})\right)\right)\\
&=O\left(\frac{r(k+r)\log n}{2n}\right)\exp(-rkt/n).
\end{split}
\end{align*}
If $\sqrt{\frac{n}{k\log n}}\le r<\frac 12\log n$, we have $k=\omega(1)$, which implies
$$\frac{r(k+r)\log n}{2n}(1+o(1)+O(1/k))=\frac{r(k+r)\log n}{2n}(1+o(1)),$$
and therefore
\begin{align*}
\begin{split}
\alpha_\lambda&=\exp(-rkt/n)\left(\exp\left(\frac{-r(k+r)\log n}{2n}(1+o(1)+O(1/k))\right)-\left(1-\frac rn+O(n^{-2+o(1)})\right)\right)\\
&=O(\exp(-rkt/n)).
\end{split}
\end{align*}
Summing up the above two cases and applying \eqref{eq: estimate of sum of d_l^2}, we have
\begin{align}
\begin{split}
\sum_{n-\l_1<\frac 12\log n}d_\l^2|\alpha_\l|^2&\le \sum_{n-\l_1<\sqrt{\frac{n}{k\log n}}}d_\l^2|\alpha_\l|^2+\sum_{\sqrt{\frac{n}{k\log n}}\le n-\l_1<\frac 12\log n}d_\l^2|\alpha_\l|^2\\
&= \sum_{r<\sqrt{\frac{n}{k\log n}}}O\left(\frac{r(k+r)\log n}{2n}\right)^2\frac{n^{2r}}{r!}\exp(-2rkt/n)\\
&+\sum_{\sqrt{\frac{n}{k\log n}}\le r<\frac 12\log n}\frac{n^{2r}}{r!}O(\exp(-2rkt/n))\\
&\le \sum_{r\ge 0}O\left(\frac{r(k+r)\log n}{2n}\right)^2\frac{\exp(-2cr)}{r!}+\sum_{r\ge\sqrt{\frac{n}{k\log n}}}O\left(\frac{\exp(-2cr)}{r!}\right)\\
&= O\left(\frac{k^2(\log n)^2}{n^2}\right)+o(1)\\
&=o(1).
\end{split}
\end{align}

For case \eqref{item: dTV for wide t}, we have
$$\frac{r(k+r)\log n}{2n}(1+o(1)+O(1/k))=O(n^{c_0-1+o(1)}),$$
and therefore
\begin{align*}
\begin{split}
\alpha_\lambda&=\exp(-rkt/n)\left(\exp(O(n^{c_0-1+o(1)}))-\left(1-\frac rn+O(n^{-2+o(1)})\right)\right)\\
&\le n^{c_0-1+o(1)}\exp(-rkt/n).
\end{split}
\end{align*}
Applying \eqref{eq: estimate of sum of d_l^2}, we have
\begin{align}
\begin{split}
\sum_{n-\l_1<\frac 12\log n}d_\l^2|\alpha_\l|^2&\le n^{2c_0-2+o(1)}\sum_{r<\frac 12\log n}\frac{n^{2r}}{r!}\exp(-2rkt/n)\\
&\le n^{2c_0-2+o(1)}\sum_{r\ge 0}\frac{\exp(-2rkt'/n)}{r!}\\
&=n^{2c_0-2+o(1)}.
\end{split}
\end{align}
Here, we use the fact that $|2kt'/n|=\log \log n-\omega(1)$, which implies $\exp\exp(-2kt'/n)=n^{o(1)}$. 
\end{enumerate}

Combining our discussion of $r\in[1,\frac 12\log n)$ and $r\in[\frac 12\log n,\log n]$ above, we complete the proof.
\end{proof}

\section{From tractable measure to total variation distance}\label{sec: From tractable measure to total variation distance}

In this section, we prove the following theorems, which elaborate on \Cref{thm: window growing m} and \Cref{thm: window fixed m} in detail.

\begin{thm}\label{thm: restate of window growing m}
Let $(n_N)_{N\ge 1},(m_N)_{N\ge 1},(l_N)_{N\ge 1},(k_N)_{N\ge 1}$, and $(t_N)_{N\ge 1}$ be sequences of positive integers that satisfy the following:
\begin{enumerate}
\item $n_N=m_N\cdot l_N$ for all $N\ge 1$.
\item $\lim_{N\rightarrow\infty}l_N=\infty$, but $\lim_{N\rightarrow\infty}\frac{l_N}{\log n_N}=0$.
\item $k_N\ge 2$ for all $N\ge 1$, and $\limsup_{N\rightarrow\infty}\log n_N k_N\le c_0$, where $c_0\in[0,1)$ is an absolute constant.
\item $t_N=\frac{n_N}{k_N}(\log n_N-\frac 12\log l_N+c)$ for all $N\ge 1$, where $c\in\R$ is an absolute constant.
\end{enumerate}
Then, we have
$$\lim_{N\rightarrow\infty}d_{TV}(\sim_{l_N}\backslash P_{n_N,k_N}^{*t_N},\sim_{l_N}\backslash U_{\S_{n_N}})=2\Phi(\exp(-c)/2)-1.$$
\end{thm}

\begin{thm}\label{thm: restate of window fixed m}
Let $l\ge 2$ be a fixed integer, $(n_N)_{N\ge 1},(m_N)_{N\ge 1},(k_N)_{N\ge 1}$, and $(t_N)_{N\ge 1}$ be sequences of positive integers that satisfy the following:
\begin{enumerate}
\item $n_N=m_N\cdot l$ for all $N\ge 1$.
\item $\lim_{N\rightarrow\infty}m_N=\infty$.
\item $k_N\ge 2$ for all $N\ge 1$, and $\lim_{N\rightarrow\infty}\frac{k_N\log n_N}{n_N}=0$.
\item $t_N=\frac{n_N}{k_N}(\log n_N+c)$ for all $N\ge 1$, where $c\in\R$ is an absolute constant.
\end{enumerate}
Then, we have
$$\lim_{N\rightarrow\infty}d_{TV}(\sim_{l}\backslash P_{n_N,k_N}^{*t_N},\sim_{l}\backslash U_{\S_{n_N}})=d_{TV}(\Pois(l+\exp(-c)),\Pois(l)).$$
\end{thm}

As in \Cref{sec: Approximating the shuffling with an explicitly tractable measure}, for notational convenience, we suppress the subscript $N$ for the rest of this section.

\begin{defi}
Let $n=ml$. We define the \emph{fixed point set} $F_{n,l}:\S_n\rightarrow 2^{[n]}$ as follows. For $\sigma\in\S_n$, let
$$F_{n,l}(\sigma):=\{i\in [n]:\ \sigma(i)\equiv i \pmod m\}.$$
The function $F_{n,l}$ naturally descends to a well-defined function on $\sim_l\backslash\S_n$, which we still denote by $F_{n,l}$. We say that an element $\sigma\in\S_n$ is \emph{clustered} if $|F_{n,l}(\sigma)\cap T_{i,l}|\ge 2$ for some $1\le i\le m$.
\end{defi}

\begin{defi}\label{defi: construction of symmetric nu}
Retain the notation and the sampling procedure from \Cref{defi: construction of nu}. We define a probability measure $\nu_{n,k}^{\mathrm{sym},t}$ on $\S_n$ by following exactly the same construction as for $\nu_{n,k}^t$, except that in the final step, after sampling $S_t$, we choose a uniformly random element of $\S_{[n]\backslash S_t}$ and then view it as an element of $\S_n$ by fixing every element of $S_t$. We denote by $\nu_{n,k}^{\mathrm{sym},t}$ the law of the resulting random permutation.
\end{defi}

Recall that $\nu_{n,k}^t$ is a measure over $\A_n$ or $\A_n^c$, but we can regard it as a measure in $\S_n$ where the extended part has zero measure. In this way, we write $\sim_l\backslash\nu_{n,k}^t$ as the pushforward measure of $\nu_{n,k}^t$ over $\sim_l\backslash\S_n$. Similarly, we write $\sim_l\backslash\nu_{n,k}^{\sym,t}$ for the pushforward measure of $\nu_{n,k}^{\sym,t}$ over $\sim_l\backslash\S_n$. When $l\ge 2$, it is clear that $\sim_l\backslash\nu_{n,k}^t$ and $\sim_l\backslash\nu_{n,k}^{\sym,t}$ are identical.

\begin{lemma}\label{lem: non cluster and Poisson}
Let $n=ml$, such that $2\le l\le n^{o(1)}$. Moreover, suppose that $x=n^{o(1)}$. Sample a uniformly random subset $T$ of size $x$ in $[n]$, then sample a uniformly random element $\sigma\in\S_{[n]\backslash T}$ and view it as an element of $\S_n$ by fixing all of the elements in $T$. Then, we have
\begin{equation}\label{eq: non cluster}
\mathbf{P}(\sigma\text{ is clustered})\le n^{-1+o(1)},
\end{equation}
and
\begin{equation}\label{eq: fixed point poisson}
d_{TV}(|F_{n,k}(\sigma)|,x+\Pois(l))\le n^{-1+o(1)}.
\end{equation}
\end{lemma}

\begin{proof}
First, we have 
\begin{equation}\label{eq: St not cluster}
\mathbf{P}(|T\cap T_{i,l}|\le 1\quad\forall 1\le i\le m)=\left(1-\frac ln\right)\left(1-\frac {2l}{n}\right)\cdots\left(1-\frac{(x-1)l}{n}\right)=1-n^{-1+o(1)}.
\end{equation}
Second, for all $1\le i\le m$ such that $T\cap T_{i,l}=\emptyset$, we have
$$
\mathbf{P}(F_{n,l}(\sigma)\cap T_{i,l}\ne\emptyset)=\left(1-\frac{l}{n-x}\right)\cdots\left(1-\frac{l}{n-x-l+1}\right)=1-\frac{l^2}{n}-n^{-2+o(1)},
$$
so that 
$$\mathbf{P}(F_{n,l}(\sigma)\cap T_{i,l}=\emptyset)=\frac{l^2}{n}+O(n^{-2+o(1)}).$$
Also, we have
$$\E[F_{n,l}(\sigma)\cap T_{i,l}]=\frac{l^2}{n-x}=\frac{l^2}{n}+O(n^{-2+o(1)}),$$
and this implies
$$\mathbf{P}(|F_{n,l}(\sigma)\cap T_{i,l}|\ge 2)\le n^{-2+o(1)}.$$
Summing up all $1\le i\le m$ such that $T\cap T_{i,l}=\emptyset$, we have
\begin{equation}\label{eq: sigma not cluster}
\mathbf{P}(|F_{n,l}(\sigma)\cap T_{i,l}|\ge 2,\quad\forall i\text{ such that }T\cap T_{i,l}=\emptyset)\le n^{-1+o(1)}.
\end{equation}
Third, for all $1\le i\le m$ such that $T\cap T_{i,l}\ne\emptyset$, we have
$$
\mathbf{P}(F_{n,l}(\sigma)\cap T_{i,l}\ne\emptyset)\ge\left(1-\frac{l}{n-x}\right)\cdots\left(1-\frac{l}{n-x-l+1}\right)=1-\frac{l^2}{n}-n^{-2+o(1)}=1-n^{-1+o(1)},
$$
and therefore
\begin{equation}\label{eq: sigma not intersect St}
\mathbf{P}(F_{n,l}(\sigma)\cap T_{i,l}\ne\emptyset,\quad\forall i\text{ such that }S_t\cap T_{i,l}\ne\emptyset)\le x\cdot n^{-1+o(1)}=n^{-1+o(1)}.
\end{equation}
The bounds in \eqref{eq: St not cluster}, \eqref{eq: sigma not cluster}, and \eqref{eq: sigma not intersect St} together complete the proof of \eqref{eq: non cluster}. The proof of \eqref{eq: fixed point poisson} is a standard Poisson approximation argument; see, for example, the papers from Arratia-Goldstein-Gordon \cite{arratia1989two,arratia1990poisson}.
\end{proof}

Now, we are ready for our proof of \Cref{thm: restate of window growing m} and \Cref{thm: restate of window fixed m}.

\begin{proof}[Proof of \Cref{thm: restate of window growing m}]
In this case, we have 
$$t'=\frac{n(c-\frac 12\log l)}{k}=-\frac{n(\log\log n-\omega(1))}{2k}$$ 
because $l=o(\log n)$. Therefore, applying the natural quotient map and using monotonicity of total variation distance under pushforward, we can obtain from case \eqref{item: dTV for wide t} of \Cref{thm: pnkt and nu} that 
$$
d_{\mathrm{TV}}(\sim_l\backslash P_{n,k}^{*t},\sim_l\backslash\nu_{n,k}^{\sym,t})=o(1).
$$
Hence it suffices to prove 
$$
d_{\mathrm{TV}}(\sim_l\backslash\nu_{n,k}^{\sym,t},\sim_l\backslash U_{\S_n})=2\Phi(\exp(-c)/2)-1+o(1).
$$
Now, let us sample $\tilde\sigma\in\sim_l\backslash \S_n$  with respect to the law of $\sim_l\backslash\nu_{n,k}^{\sym,t}$, and sample $\tilde\tau\in\sim_l\backslash \S_n$ with respect to the law of $\sim_l\backslash U_{\S_n}$.
Observe that for all $T\subseteq[n]$, the conditional law of $\tilde\sigma$ given the event $F_{n,k}(\tilde\sigma)=T$ is the same as the conditional law of $\tilde\tau$ under the event $F_{n,k}(\tilde\tau)=T$. Therefore, it suffices to prove
$$d_{TV}(F_{n,k}(\tilde\sigma),F_{n,k}(\tilde\tau))=2\Phi(\exp(-c)/2)-1+o(1).$$
Recall from our definition that $\gamma_t=\exp(-c+\frac 12\log l)=\sqrt{l}\exp(-c)$. Hence, applying \eqref{eq: fixed point poisson}, we have
$$\mathbf{P}(|F_{n,k}(\tilde\sigma)|\ge\log n)\le n^{-1+o(1)}+\mathbf{P}\left(\Pois\left(l+\sqrt l\exp(-c)\right)\ge \log n\right)= n^{-1+o(1)}, $$
$$\mathbf{P}(|F_{n,k}(\tilde\tau)|\ge\log n)\le n^{-1+o(1)}+\mathbf{P}(\Pois(l)\ge \log n)= n^{-1+o(1)}. $$
Therefore, by natural truncation it follows from \eqref{eq: non cluster} and \eqref{eq: fixed point poisson} that
\begin{align}
\begin{split}
d_{TV}(F_{n,k}(\tilde\sigma),F_{n,k}(\tilde\tau))&=d_{TV}(|F_{n,k}(\tilde\sigma)|,| F_{n,k}(\tilde\tau)|)+o(1)\\
&=d_{TV}\left(\Pois\left(l+\sqrt l\exp(-c)\right),\Pois(l)\right)+o(1)\\
&=2\Phi(\exp(-c)/2)-1+o(1).\\
\end{split}
\end{align}
Here, the third line is a straightforward consequence of $l=\omega(1)$, together with the one-crossing property and the classical local limit approximation for Poisson laws; see, for example, standard references on limit theorems such as Petrov \cite[Chapter VII]{petrov1975sums}.
\end{proof}

\begin{proof}[Proof of \Cref{thm: restate of window fixed m}]
Applying the natural quotient map and using monotonicity of total variation distance under pushforward, we obtain from case \eqref{item: dTV for large k} of \Cref{thm: pnkt and nu} that 
$$
d_{\mathrm{TV}}(\sim_l\backslash P_{n,k}^{*t},\sim_l\backslash\nu_{n,k}^{\sym,t})=o(1).
$$
Hence it suffices to prove 
$$
d_{\mathrm{TV}}(\sim_l\backslash\nu_{n,k}^{\sym,t},\sim_l\backslash U_{\S_n})
=
d_{\mathrm{TV}}(\Pois(l+\exp(-c)),\Pois(l))+o(1).
$$
Now, let us sample $\tilde\sigma\in\sim_l\backslash \S_n$  with respect to the law of $\sim_l\backslash\nu_{n,k}^{\sym,t}$, and sample $\tilde\tau\in\sim_l\backslash \S_n$ with respect to the law of $\sim_l\backslash U_{\S_n}$.
Observe that for all $T\subseteq[n]$, the conditional law of $\tilde\sigma$ under the event $F_{n,k}(\tilde\sigma)=T$ is the same as the conditional law of $\tilde\tau$ given the event $F_{n,k}(\tilde\tau)=T$. Therefore, it suffices to prove
$$d_{TV}(F_{n,k}(\tilde\sigma),F_{n,k}(\tilde\tau))=d_{\mathrm{TV}}(\Pois(l+\exp(-c)),\Pois(l))+o(1).$$
Recall from our definition that $\gamma_t=\exp(-c)$. Hence, applying \eqref{eq: fixed point poisson}, we have
$$\mathbf{P}(|F_{n,k}(\tilde\sigma)|\ge\log n)\le n^{-1+o(1)}+\mathbf{P}(\Pois(l+\exp(-c))\ge \log n)= n^{-1+o(1)}, $$
$$\mathbf{P}(|F_{n,k}(\tilde\tau)|\ge\log n)\le n^{-1+o(1)}+\mathbf{P}(\Pois(l)\ge \log n)= n^{-1+o(1)}. $$
Therefore, by natural truncation it follows from \eqref{eq: non cluster} and \eqref{eq: fixed point poisson} that
\begin{align}
\begin{split}
d_{TV}(F_{n,k}(\tilde\sigma),F_{n,k}(\tilde\tau))&=d_{TV}(|F_{n,k}(\tilde\sigma)|,| F_{n,k}(\tilde\tau)|)+o(1)\\
&=d_{TV}(\Pois(l+\exp(-c)),\Pois(l))+o(1).
\end{split}
\end{align}
This completes the proof.
\end{proof}

\bibliographystyle{plain}
\bibliography{references.bib}

\end{document}